\documentclass[]{hdsr}

\graphicspath{{figs/}}

\usepackage{amsmath, amsthm,amsfonts,amssymb,bbold}
\usepackage{mathrsfs}
\usepackage{tikz-cd}
\usepackage[all]{xy}% xymatrix for CD's
\newtheorem*{theorem}{Theorem}
\newtheorem*{proposition}{Proposition}

\theoremstyle{remark}
\newtheorem*{remark}{Remark}

\newcommand{\R}{\mathbb{R}}

\newcommand{\N}{\mathbb{N}}

\newcommand{\EE}{\mathbb{E}}
\newcommand{\Pb}{\mathbb{P}}

\newcommand{\cS}{\mathcal{S}}

\def\S{\mathbb{S}}

\usepackage{mathtools}
\DeclarePairedDelimiterX{\norm}[1]{\lVert}{\rVert}{#1}

\DeclareMathOperator{\Diff}{Diff}

\DeclareMathOperator{\SE}{SE}

\DeclareMathOperator{\Tr}{Tr}
\DeclareMathOperator{\Id}{Id}
\DeclareMathOperator{\var}{var}

\newcommand{\kSqE}{k_{\text{SE}}}
\newcommand{\kMat}{k_{\text{Mat\'ern}}}
\newcommand{\gSqE}{g_{\text{SE}}}
\newcommand{\gMat}{g_{\text{Mat\'ern}}}

\usepackage{xcolor}

\begin{document}

% Larger bottom margin for the first page
\newgeometry{bottom=1.5in}

% Editorial staff will replace the following values:
% 1. Volume number
% 2. Issue number
% 3. Article DOI
% e.g. for Volume 2, Issue 3, DOI 12.345:
% \volumeheader{2}{3}{12.345}
%\volumeheader{0}{0}{00.000}

\begin{center}

  \title{Stochastics of shapes and Kunita flows}
  \maketitle

  % Start page numbering on second page. Must appear *after* \maketitle
  \thispagestyle{empty}
  
  \vspace*{.2in}

  % Authors and Affiliations
  \begin{tabular}{cc}
    Stefan Sommer\upstairs{\affilone}, Gefan Yang\upstairs{\affilone}, Elizabeth Louise Baker\upstairs{\affilone}
   \\[0.25ex]
   {\small \upstairs{\affilone} Department of Computer Science, University of Copenhagen} \\
  \end{tabular}
  
  % Replace with corresponding author email address
  \emails{
    \upstairs{*}sommer@di.ku.dk
    }
  \vspace*{0.4in}

\begin{abstract}
Stochastic processes of evolving shapes are used in applications including evolutionary biology, where morphology changes stochastically as a function of evolutionary processes. Due to the non-linear and often infinite-dimensional nature of shape spaces, the mathematical construction of suitable stochastic shape processes is far from immediate. We define and formalize properties that stochastic shape processes should ideally satisfy to be compatible with the shape structure, and we link this to Kunita flows that, when acting on shape spaces, induce stochastic processes that satisfy these criteria by their construction. We couple this with a survey of other relevant shape stochastic processes and show how bridge sampling techniques can be used to condition shape stochastic processes on observed data thereby allowing for statistical inference of parameters of the stochastic dynamics.
\end{abstract}
\end{center}

\vspace*{0.15in}
\hspace{10pt}
  \small	
  \textbf{\textit{Keywords: }} {shape stochastics, Kunita flows, bridge sampling, outer shape models, inner shape models}
  
\copyrightnotice

\section{Introduction}
Statistics of shapes has been treated extensively in the literature, both because of the mathematical intricacies resulting from working with shapes as inherently non-linear objects, and because of the many applications in e.g. biology and medical imaging. Shape spaces span from landmark based representation by Kendall (\cite{kendallShapeManifoldsProcrustean1984}) over elastic metrics (\cite{mioShapePlaneElastic2007,younesMetricShapeSpace2008}) to infinite-dimensional diffeomorphism based representations (\cite{younesShapesDiffeomorphisms2010}) with many additional models and shape spaces in between. This span includes many perspectives on how shapes should be modeled, how metrics on shape spaces can be constructed, and how statistics of shapes can be performed.

In this paper, we look at stochastics of shapes. This is interesting again from several perspectives. First, defining infinite-dimensional stochastic processes in non-linear spaces and subsequently attempting to perform statistical inference of parameters of the model from observed data is challenging. Second, shape stochastics is relevant in many applications, including evolutionary biology, where the stochastic evolution of morphology is an important aspect of the evolutionary process. Thirdly, the non-linear structure of shape spaces makes defining parametric families of probability distributions difficult. Using stochastic processes to generate distributions on shape spaces is one possible solution to this problem.

The paper builds on ideas from \cite{sommerStochasticFlowsShape2021} and \cite{stroustrupStochasticPhylogeneticModels2025}, but here we approach shape stochastics from an axiomatic perspective defining properties that shape stochastic processes should ideally satisfy. We then link this to Kunita flows (\cite{kunitaLecturesStochasticFlows1986,kunitaStochasticFlowsStochastic1997}) that satisfy these criteria by their construction. Our point of formalizing the properties and showing that Kunita flows satisfy these properties highlights Kunita flows as a natural choice of stochastic processes for shapes. This viewpoint allows for example to define a notion of variance for shape processes that is not tied to the particular choice of shape representation.
Additionally, we survey other relevant shape stochastic processes, and we describe how modern bridge sampling techniques can be used to condition shape stochastic processes on observed data and subsequently perform statistics.

With the paper, we wish to describe stochastic processes on shape spaces from an axiomatic perspective, and highlight Kunita flows as a natural family of processes that satisfy these properties. To set the stage, we first briefly describe two important families of shape spaces that we focus on in the paper. We also describe other shape stochastic processes besides Kunita flows and briefly mention other time-series models on manifolds, but we focus the paper on the axiomatic approach to shape stochastics and the link to Kunita flows.

\begin{figure}[h]
    \centering
    \includegraphics[width=0.8\textwidth,clip,trim={0 0 0 50}]{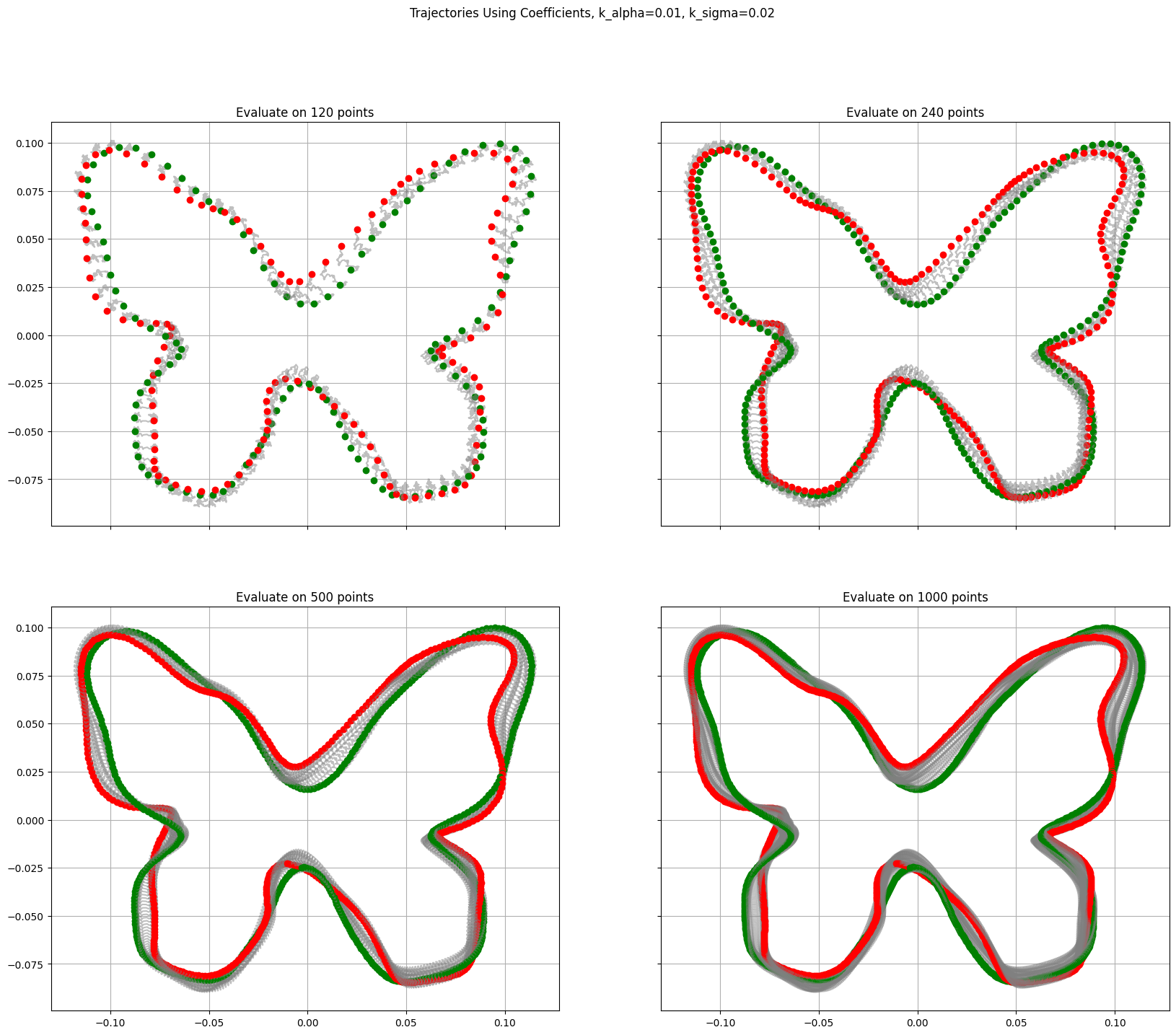}
    \caption{The outline of a butterfly acted upon by a Kunita flow of diffeomorphisms. The continous nature of the flow allows evaluation for an arbitrary discretization and hence number of points representing the shape. In the paper, we outline the Kunita flow theory and its use in the shape context.} % Add caption if needed
    \label{fig:butterflies}
\end{figure}

\subsection{Outline}
In section \ref{sec:shape_spaces}, we survey constructions of specific shape spaces as used in this paper. In section \ref{sec:axiomatic_approach}, we define properties that shape stochastic processes should ideally satisfy, and, in section \ref{sec:kunita_flows}, we link this to Kunita flows that satisfy these criteria by their construction. In section \ref{sec:other_stochastic_processes}, we survey other relevant shape stochastic processes. In section \ref{sec:conditioning_bridge_sampling}, we discuss how modern bridge sampling techniques can be used to condition shape stochastic processes on observed data, thus allowing for statistical inference from observed shape data.

\section{Shape spaces}
\label{sec:shape_spaces}
Shapes do not support the natural operations that one associates with vector spaces, which necessitates care when defining spaces of shapes, both with respect to topological, metric, and geometric properties. Here, we outline two of the most common viewpoints, outer and inner shape models, as these are in focus in the remainder of the paper. A more detailed discussion of common shape spaces is given in \cite{bauerOverviewGeometriesShape2014}.

\subsection{Outer shape models}
\label{sec:outer_shape_models}
Outer shape models take their foundation in the pattern approach of Grenander (\cite{grenanderGeneralPatternTheory1994}) and the Large deformation diffeomorphic metric mapping (LDDMM) framework of Trouve, Younes and Joshi (\cite{trouveInfiniteDimensionalGroup1995,younesComputableElasticDistances1998,joshiLandmarkMatchingLarge2000}). The name stems from shapes being subsets of a larger manifold, typically $\R^d$, $d=2,3$, and deformations of this entire space acting on the shape to produce orbit shape spaces. Specifically, we start with an appropriate subgroup $G\subset\Diff(\R^d)$ of diffeomorphisms of $\R^d$ and define a left action of elements of $G$ on the shape space $\cS$. An example of this is for configurations of distinct landmarks $x= (x_1,\ldots,x_n)$, $x_i\in\R^d$, where the action is $\phi.x= (\phi(x_1),\ldots,\phi(x_n))$ for $\phi\in G$. Similarly, for shapes represented as functions $s:M\to\R^d$ of a (compact) domain manifold $M$ to $\R^d$, the action is also by composition $\phi.s=\phi\circ s$. Typical examples of this are $M=\S^1$ for closed curves and $M=\S^2$ for surfaces. For images $I:\R^d\to\R$ (grayscale) or $I:\R^d\to\R^3$ (color), the action is by inverse composition from the right, i.e., $\phi.I=I\circ\phi^{-1}$.

Let $G$ be an appropriate such subgroup of diffeomorphisms, e.g. a subgroup of diffeomorphisms with bounded difference to the identity including all derivatives, or with $H^\infty$ difference to the identity (\cite{michorZooDiffeomorphismGroups2013}) so that $G$ is an (infinite-dimensional) Lie group. We can then define a Riemannian metric on $G$ from a quadratic form on the Lie algebra $\mathfrak{g}$, extending it to a metric on $G$ by right-invariance. Since $\mathfrak{g}$ can be associated with vector fields on $\R^d$, the metric is generally chosen of the form $\left<v,w\right>=\left<Lv,w\right>_{L^2(\R^d)}$, where $L$ is a linear operator, denoted the inertia operator. Sobolev metrics arise in this form when $L$ is a differential operator. Geodesics on $G$ for such right-invariant metrics are governed by the EPDiff equations (\cite{holmSolitonDynamicsComputational2004a}).

Because of the right-invariance of the resulting metric, it induces a Riemannian structure on $\cS$ so that the map $g\mapsto g.s_0$ for fixed $s_0\in\cS$ becomes a Riemannian submersion. One case often considered is the landmark configuration shape spaces $\cS_n=\{(x_1,\ldots,x_n)\in\R^d|x_i\neq x_j\text{ for }i\neq j\}$, where the induced metric takes an explicit form in terms of the cometric 
\begin{equation*}
    \left<p,\tilde{p}\right>_x = \sum_{i,j=1}^n p_i^T k(x_i,x_j)\tilde{p}_j
\end{equation*}
for covectors $p,\tilde{p}\in(\R^d)^n$ and a Green's kernel $k:\R^d\times\R^d\to\R^d\otimes\R^d$ of $L$. Geodesic equations for such landmark equations can be found by defining the Hamiltonian $H(x,p)=\frac12\left<p,p\right>_x$ using the cometric and solving the Hamilton-Jacobi equation
\begin{equation}
    \begin{split}
    dx_{t,i} &= \nabla_{p_{t,i}}H(x_t,p_t)dt = \sum_{j=1}^n k(x_i,x_j)p_{t,j}dt \\
    dp_{t,i} &= -\nabla_{x_{t,i}}H(x_t,p_t)dt = -\sum_{j=1}^n D_1(k(x_i,x_j)p_{t,j})^Tp_{t,i}dt .
    \end{split}
    \label{eq:landmark_geodesic}
\end{equation}
The initial conditions consists of the initial position $x_0$ and initial momentum $p_0$. The landmark configuration evolution $x_t$ in solutions to \eqref{eq:landmark_geodesic} can also be found by solving a geodesic equation on $G$ with appropriate initial conditions to get a solution $\phi_t$ and setting $x_t=\phi_t.x_0$. Here $\phi_t$ is a time-parametrized family of diffeomorphisms inducing the same dynamics on $\cS_n$ as the system \eqref{eq:landmark_geodesic}. As such, the landmark geodesic equations are inspiration for properties of stochastic processes on shape spaces we define below: The dynamics arise from shape dynamics in the form of $\phi_t$ that are not discretized and independent on the choice of landmarks, and dynamics on $G$ can be recovered from landmark dynamics in the limit when increasing the number of landmarks. We will see stochastic landmark dynamics in section \ref{sec:kunita_flows} using Kunita flows, and later, in section \ref{sec:other_stochastic_processes}, we will describe examples of stochastic perturbations of the landmark geodesic equations.

The above described relation between $x_t$ and $\phi_t$ can be realized from a horizontal lift from $\cS_n$ to $G$. To exemplify this when the number of landmarks is changed, consider a geodesic $x_t$ on a landmark space $\cS_n$ with initial momentum $p_0$ where e.g. the last $l$ components of $p_0$ are zero. The horizontal lift $\phi_t$ of $x_t$ to $G$ is a geodesic on $G$ with initial momentum supported on $n-l$ Dirac measures $\delta_{x_{0,i}}$, $i=1,\ldots,n-l$. This geodesic could equivalently be realized from a horizontal lift of a corresponding landmark configuration in $\cS_{n-l}$. The trajectories of the two geodesics on $\cS_n$ and $\cS_{n-l}$ respectively will in this case be the same for the first $n-l$ landmarks. One can concretely see this from the Hamiltonian equations \eqref{eq:landmark_geodesic} by considering the effect of the last $l$ components of $p_0$ being $0$.
Similarly, adding landmarks with zero initial momentum will not change the dynamics. In general, however, lifts of geodesics from $\cS_n$ to $G$ can have initial conditions supported on $n$ Diracs whereas lifts of geodesics from $\cS_{n-l}$ to $G$ will have initial conditions supported on at most $n-l$ Diracs so geodesics can be different in the two cases. 

When a curve $\phi_t$ on $G$ is already defined, the induced curve $x_t=\phi_t.x_0$ on $\cS_n$ is independent of the number of landmarks in the sense that, for $l\le n$ with $\tilde{x}_0=(x_{0,1},\ldots,x_{0,l})$, the curve $\tilde{x}_t=\phi_t.\tilde{x}_0$ is the same as the curve $x_t$ for the first $l$ landmarks. This will specifically be the case for the stochastic Kunita flows that we consider below. In these cases, trajectories are defined individually from $\phi_t$ acting on each landmark, and the landmarks do not ``sense'' the number of landmarks. This is in contrast to other stochastic processes on shape spaces, e.g. the Riemannian Brownian motion described in section~\ref{sec:other_stochastic_processes}, where the landmarks interact and changing the number of landmarks will change the dynamics.

\subsection{Inner shape models}
In contrast to outer shape models, inner shape models are defined directly on the shape space instead of referring to diffeomorphisms of an embedding space. Let $M$ be a compact manifold and consider spaces of parametrized shapes including immersions $\mathrm{Imm}(M,\R^d)$ and embeddings $\mathrm{Emb}(M,\R^d)$. Typical examples would be closed curves $s:M\to\R^d$ with $M=\S^1$ and closed surfaces $s:M\to\R^d$, $M=\S^2$.
These spaces can be equipped with Sobolev metrics of different orders including elastic metrics (\cite{mioShapePlaneElastic2007,younesMetricShapeSpace2008,jermynElasticShapeMatching2012,jermynElasticShapeAnalysis2017}) and $H^2$-metrics (\cite{hartmanElasticShapeAnalysis2023}). The shapes are parametrized since composing with a diffeomorphism of $M$ will result in a different shape, even though the subset $s(M)$ is the same. Quotienting out the action of diffeomorphisms of $M$ results in the (non-parametrized) shape spaces $\mathrm{Imm}(M,\R^d)/\Diff(M)$. If the metric is invariant under the action, this results in a metric on the shape space being induced from the metric on $\mathrm{Imm}(M,\R^d)$. The space of parametrized shapes is often denoted a pre-shape space, while the quotient is called a shape space.

There is extensive literature on inner shape models, and more details can be found in e.g. \cite{michorOverviewRiemannianMetrics2007}. We will see examples of stochastic processes defined directly on inner shape models in section \ref{sec:other_stochastic_processes}. Since diffeomorphisms act on curves and surfaces, Kunita flows also induce stochastic processes on inner shape spaces.

\section{An axiomatic approach to shape stochastics}
\label{sec:axiomatic_approach}
We here take an axiomatic approach to shape stochastics. In the present context, we define a shape stochastic process as a stochastic process that
\begin{enumerate}
    \item is defined independently of the specific representation or discretization of the shape.
    \label{item:representation_consistency}
    \item preserves shape structure: When the process is started from a shape, it should keep being a shape for any finite time. 
    \label{item:shape_structure_preservation}
    \item is equivariant to the action of rigid transformations groups and reparameterization groups: The dynamics should not change if the shape is transformed by rigid transformations or if it is reparameterized, thus allowing the dynamics to descend to quotients of (pre-)shape spaces.
    \label{item:rigid_transformation_invariance}
    \item can be recovered from discretizations: The process should be recoverable from discretized processes in the limit as the discretization is refined.
    \label{item:discretization_recovery}
\end{enumerate}

Below, we discuss and formalize these properties. See also Figure~\ref{fig:shape_process_properties} for a schematic illustration of the intuition behind the properties. We assume the processes are defined on a standard probability space $(\Omega, \mathcal{F}, \Pb)$. As we in the end are interested in the statistical properties of the shape processes, the convergence properties are formalized as statements in distribution. When we have a map $f:\cS\to\cS_1$ between two shape spaces, we will use the notation $f_*s_t$ for the pushforward of a process $s_t$ in $\cS$ by $f$. For each $t$, this means that the random variable $s_t$ is mapped to the random variable $f\circ s_t$, i.e. we just apply $f$ to the shape process. The law of $f_*s_t$ is then $\Pb(f_*s_t\in A) = \Pb(s_t\in f^{-1}(A))$ for measurable subsets $A\subset\cS_1$.

\begin{figure}[h]
    \centering
    \includegraphics[width=.9\textwidth]{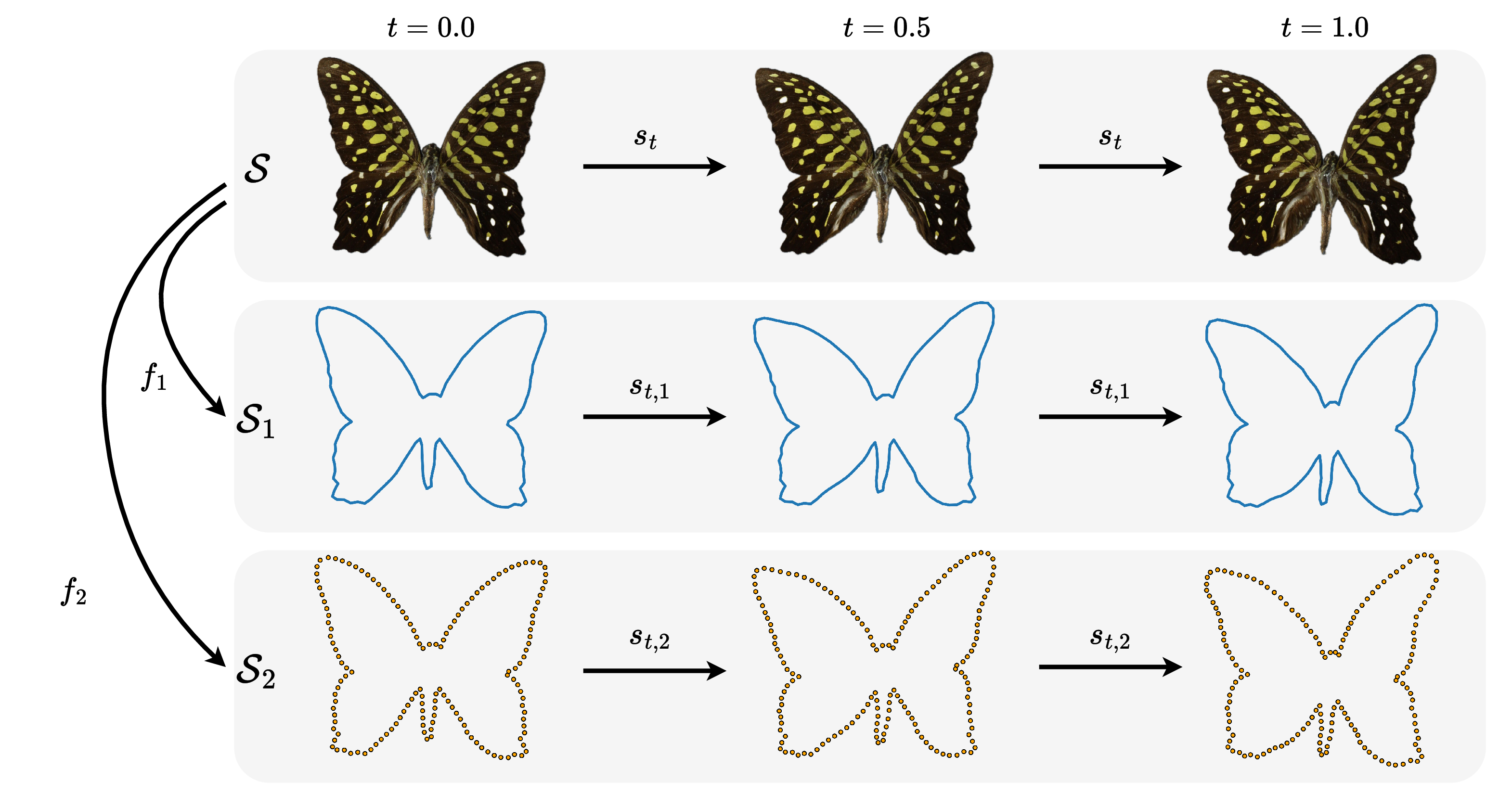}
    \hspace{10pt}
    \caption{Schematic illustration of the axiomatic properties of shape stochastic processes. The first row shows a butterfly changing shape with a time-continuous stochastic process. The second row shows the shape process represented by the outline of the butterfly as a curve, and the third row shows the process represented as discrete points on the outline. The consistency between the top row and the bottom rows illustrates representation independence (property \ref{item:representation_consistency}), while the preservation of the curve structure across time illustrates shape structure preservation (property \ref{item:shape_structure_preservation}). The discretization recovery property (property \ref{item:discretization_recovery}) is illustrated by the fact that the discrete points converge to the continuous curve as the discretization is refined. The shapes are all produced by the action of a Kunita flow on the butterfly image at $t=0$.}
    \label{fig:shape_process_properties}
\end{figure}

\subsection{Representation and discretization independence}
A common approach to define shape dynamics is to define the processes directly on the representation space, for example on landmark configurations by perturbing each landmark. However, this often leads to a dependence on the representation, e.g. the specific number of landmarks. Whether a surface is represented as a surface, by a set of landmarks on the surface, or as an image from which the surface is segmented, the dynamics should ideally be equivalent up to the loss of information specific to each representation. The first property covers this representation independence.

Let $\cS$ be a shape space and $s_t$ a stochastic process in $\cS$. We say that a process $s_{1,t}$ in a shape space $\cS_1$ is a representation of $s_t$ if $s_{1,t}$ is a pushforward of $s_t$, i.e. there exists a map $f_1:\cS\to\cS_1$ such that $s_{1,t}=(f_1)_*s_t$. With this, we define a shape process $s_{1,t}$ to be independent of its specific representation if there exists a shape process $s_t$ defined without using properties of $\cS_1$ or $f_1$ such that $s_{1,t}$ is a representation of $s_t$.

We can think of $s_t$ as the underlying shape process and $s_{1,t}$ as particular representation, possibly a discrete, numerical representation. If we have two different representations $s_{1,t}$ and $s_{2,t}$, they will both be pushforwards of the same underlying shape, and we are ensured that the dynamics represent $s_t$ to the extent the maps $f_1,f_2$ preserve the shape properties. 

At this point, we do not require the representation to enable recovery of the underlying shape process to a specific degree, i.e. the map $f_1$ can in principle be trivial. The ability of the representation to actually represent the underlying shape process is captured by property (\ref{item:discretization_recovery}).

\subsection{Shape structure preservation}
We define a stochastic process $s_t$ with $s_0\in\cS$ to be shape structure preserving if, for all $t\geq 0$, $s_t\in\cS$.

Shape often represent continous embedded objects, e.g. curves and surfaces. For such shapes, the property implies that, for embedded shapes in non-trivial situations,
    \begin{itemize}
        \item[1.] close points on the shape will be correlated: Infinitely close points moving independently cannot preserve continuity.
        \label{item:shape_structure_preservation_correlated}
        \item[2.] the process will be state-dependent: Since the shape will change, points initially having independent dynamics can move close together and thus be correlated because of 1.
        \label{item:shape_structure_preservation_state_dependent}
    \end{itemize}

\subsection{Rigid transformation and reparametrization equivariance}
Let $\SE(d)$ be the special Euclidean group of translations and rotations acting on the shape space $\cS$ with action $g.s$, $g\in\SE(d)$. Let $s_t^{s_0}$ be a stochastic process in $\cS$ started at $s_0\in\cS$. The process is equivariant to the action of $\SE(d)$ on $s_0$ if the law of $s_t^{g.s_0}$ is the same as the law of $g.s_t^{s_0}$ for $g\in\SE(d)$.

Similarly, let $\Diff(M)$ be the group of diffeomorphisms of a manifold $M$ with the right action $s.g$, $g\in\Diff(M)$. We say that the process is equivariant to the action of $\Diff(M)$ if the law of $s_t^{s_0.g}$ is the same as the law of $s_t^{s_0}.g$ for $s_0\in\cS$ and $g\in\Diff(M)$. This property is relevant for inner shape models $s:M\to\R^d$, where the action of $\Diff(M)$ is reparameterizing the shape.

\subsection{Recovery from discretizations}
Let $\cS$ be a shape space, let $\mathcal P$ denote a class of finite dimensional discretizations, and let $P\in\mathcal P$ be a specific discretization, e.g. a representation of the shapes at a finite number of points.
Let $\cS_P$ represent the discretized shape space with map $f_P:\cS\to\cS_P$ discretizing the shapes. For a shape process $s_t$, we say that the discretizations $s_{P,t}=f_P(s_t)$ allows $s_t$ to be recovered if the law of $s_t$ is determined when $s_{P,t}$ is known for all such finite dimensional representations $P\in\mathcal P$. 

\begin{remark}[Deterministic case, outer shape models]
While the properties above are defined for and focus on stochastic processes, we can also consider the case of deterministic shape changes, in particular since the outer shape models with action of diffeomorphisms serve as inspiration for the properties. 

If we consider the landmark shape spaces $\cS_n$ as orbit spaces of a group $G$ and with induced metric as described in section~\ref{sec:outer_shape_models}, geodesics on $\cS_n$ are projections of geodesics in $G$. A geodesic $x_t$ in $\cS_n$ is thus a representation of a geodesic $\phi_t$ in the sense of proerty (\ref{item:representation_consistency}): $
\phi_t$ maps to $x_t$ via the action $x_t=\phi_t.x_0$. For sufficiently smooth metrics on $G$, $x_t$ will stay in $\cS_n$ for all $t\geq 0$ and property (\ref{item:shape_structure_preservation}) will hold. Specifically, for such metrics, $\cS_n$ will be geodesically complete (\cite{bauerOverviewGeometriesShape2014,habermannCharacterizationGeodesicCompleteness2025}) and landmark configurations changing geodesically will not see landmarks collide in finite time. The rigid transformation and reparameterization equivariance property (\ref{item:rigid_transformation_invariance}) holds for suitably invariant metrics analogously to the stochastic case as we will see below. Furthermore, geodesics on $\phi_t$ can be recovered from geodesics on $\cS_n$ for increasing $n$ under mild conditions if the landmarks are sufficiently dense in the domain (\cite{jacobsHigherorderSpatialAccuracy2014}). Therefore, the outer shape models can be seen as a deterministic inspiration for the case of stochastic processes.
\end{remark}

\section{Kunita flows}
\label{sec:kunita_flows}
We now focus on a concrete example of a family of processes satisfying the properties (\ref{item:representation_consistency})-(\ref{item:discretization_recovery}), shape processes induced by Kunita flows.

Kunita flows are stochastic processes of diffeomorphisms $\phi_t\in\Diff(\R^d)$ such that the increments of point trajectories $\phi_t(x),\phi_t(y)$, $x,y\in\R^d$ are correlated with local characteristics that are at least continuous, and where the drift of the point trajectories is also spatially continuous. Below, we outline the construction following \cite{kunitaLecturesStochasticFlows1986,kunitaStochasticFlowsStochastic1997}, and we refer to these works for full details.

Let $(\Omega,\mathcal F,\Pb)$ be a standard probability space. Let $\phi_{s,t}:\R^d\times\Omega\to\R^d$, $0\leq s\leq t$ be an indexed family of random fields where, for each $s,t,\omega$, the map $\phi_{s,t}(\cdot,\omega):\R^d\to\R^d$ is measurable. We will skip $\omega$ in the notation of those maps so that $\phi_{s,t}$ are random maps $\R^d\to\R^d$. Then the family is called a stochastic flow if 1) $s,t,x\mapsto\phi_{s,t}(x)$ is continuous in $s,t,x$ in probability, 2) $\phi_{s,s}=\Id_{\R^d}$ for all $s$, and 3) $\phi_{s,t}\circ\phi_{t,u}=\phi_{s,u}$ for all $s\leq t\leq u$. If, in addition, all increments $\phi_{t_i,t_{i+1}}$ are independent for any $0\leq t_0\leq t_1\leq\cdots\leq t_n$ and $\phi_{s,t}(x)$ is continuous in $t$ a.s., then $\phi_{s,t}$ is called a Brownian flow.

Under additional technical assumptions, $\phi_{s,t}$ can be shown to be a Brownian flow of homeomorphisms meaning that 1) $\phi_{s,t}(x)$ is continuous in $s,t,x$ a.s.; and 2) $\phi_{s,t}$ is a homeomorphism a.s. If $\phi_{s,t}$ is a $C^l$, $l\in\N$, diffeomorphism a.s., then $\phi_{s,t}$ is called a Brownian flow of $C^l$-diffeomorphisms.

For Brownian flows, the entities
\begin{align*}
    b_t(x) &= \lim_{h\downarrow 0}\frac1h\EE[\phi_{t,t+h}(x)-x], \\
    a_t(x,y) &= \lim_{h\downarrow 0}\frac1h\EE[(\phi_{t,t+h}(x)-x)(\phi_{t,t+h}(y)-y)^T]
\end{align*}
are denoted the local characteristics of the flow. For the present purposes, we can denote them the drift and the covariance kernel of the flow, respectively. Kunita then shows the following smoothness characterization of a stochastic flow of homeomorphisms from the local characteristics.

\begin{theorem}
    Let $\phi_{s,t}$ be a Brownian flow of homeomorphisms. If $b_t$ and $a_t$ are $C^l$ in $x$ and $y$ with bounded $l$th order derivatives in both $x$ and $y$ simultaneously, then $\phi_{s,t}$ is a Brownian flow of $C^{l-1}$-diffeomorphisms.  
    \label{thm:kunita_smoothness}
\end{theorem}

\subsection{Semimartingales and stochastic integrals}
To construct Brownian flows of diffeomorphisms from stochastic differential equations (SDEs), we first need to discuss martingales and semimartingales in this context. 
A random field taking values in $C(\R^d,\R^d)$ is called $C$-valued. A $C$-valued random field $M_t$ is called a $C$-valued martingale if $M_t(x)$ is an $\R^d$-valued martingale for each $x$. If, for any multi-index $\alpha$ with $|\alpha| \leq l$, $D_x^\alpha M_t$ is a $C$-valued martingale, then $M_t$ is called a $C^l$-valued martingale.

Let $S_t$ be a continuous, $C$-valued random field of the form
\begin{equation*}
    S_t = S_0 + M_t + V_t
\end{equation*}
with $C$-valued martingale $M_t$, for each $x$, $V_t(x)$ is a process of bounded variation, and $M_0(x)=0, V_0(x) = 0$. Then $S_t$ is denoted a $C$-valued semimartingale. Similarly to martingales, if $S_t$ is a $C^l$-valued random field,  $M_t$ is a $C^l$-valued martingale and $D_x^\alpha V_t(x), |\alpha| \leq l$ is of bounded variation, then $S_t$ is called a $C^l$-valued semimartingale.

Let $a_t,b_t$ be such that
\begin{align*}
    V(x,t) &= \int_0^t b_r(x)dr, \\
    \left< M(x,t), M(y,t)^* \right> &= \int_0^t a_r(x,y)dr
    \ .
\end{align*}
Then $a_t,b_t$ are called the local characteristics of the semimartingale.

From $S_t$, one obtains the stochastic differential equation $d\phi_t=S_{dt}(\phi_t)$ of which $\phi_t$ is a solution if it satisfies
\begin{equation}
    \phi_t(x) = x + \int_s^t S_{dr}(\phi_r(x))
    \label{eq:stochastic_integral}
\end{equation}
where $\int_s^t S_{dr}(\phi_r(x))$ is the stochastic integral of $S_t$ with respect to $\phi_t$. See \cite{kunitaLecturesStochasticFlows1986} for the concrete construction of this integral. With this, Kunita shows the following results.
\begin{theorem}
    If the local characteristics $a_t,b_t$ are Lipschitz continuous with linear growth, then \eqref{eq:stochastic_integral} has a unique solution with a modification that is a stochastic flow of diffeomorphisms.
\end{theorem}

\subsection{Hilbert space formulation}
We can equivalently think of the map $X_t$ defined by $x\mapsto \phi_t(x)-x$ as an element of the Hilbert space $L^2(D,\R^d)$ with $D$ a bounded subset of $\R^d$. This is similar to the definition of the subgroup $G$ in section \ref{sec:shape_spaces} where the difference between a diffeomorphism and the identity is an element of a Hilbert space. $X_t$ models the displacement of $\phi_t$ from the identity.  Let $W_t$ be a cylindrical Wiener process on $L^2(D,\R^d)$ and consider the Hilbert space SDE 
\begin{equation}
    dX_t=Q^{1/2}(X_t)dW_t \label{eq:kunita_flow_hilbert} 
\end{equation}
where the diffusion operator $Q^{1/2}(X_t)$ on $L^2(D,\R^d)$ is given by
\begin{equation}
Q^{1/2}(X_t)(f)(x)
    =\int_{D}k(\phi_t(x),\zeta)f(\zeta)d\zeta
    =\int_{D}k(X_t(x)+x,\zeta)f(\zeta)d\zeta
    \label{eq:kunita_flow_hilbert_operator}
\end{equation}
for a kernel $k:D\times D\to\R^d\otimes\R^d$, $f\in L^2(D, \R^d)$ and $x \in \R^d$. Being an integral operator, there exists a square of the form
\begin{equation}
    Q(f)(x) = \int_{D}g(x,\zeta)f(\zeta)d\zeta
    \label{eq:kunita_flow_hilbert_operator_square}
\end{equation}
with $g(x,y)=\int_D k(x,\zeta)k(\zeta,y)^T d\zeta$. The corresponding local characteristics are then
\begin{align*}
    b_t(x)&=0 \\
    a_t(x,y)&=g(x,y).
\end{align*}
This formulation gives a direct correspondence between the kernel $k$, its square $g$ and the covariance characteristic $a_t$ of the flow. The kernel $k$ can be chosen to satisfy the required smoothness conditions to ensure that the resulting flow is a flow of diffeomorphisms. As $k$ is often translation and rotation invariant and scalar valued, one can write $k(x,y)=k(r)\Id_d$ where $r=\|x-y\|$ and $k$ is now interpreted as a function $k:\R\to\R$.

One choice of kernel is the squared exponential kernel 
\begin{equation}
\kSqE(r)=\alpha e^{-\frac{r^2}{2\sigma^2}}
\label{eq:gaussian_kernel}
\end{equation}
which results in $\gSqE(r)=\alpha^2\pi^{d/2}\sigma^d e^{-\frac{r^2}{4\sigma^2}}$ being a squared exponential as well. Another common choice are Green's functions of Laplacian and higher order Sobolev operators. For such an operator $L=\alpha^{-1}(\Id-\sigma^2\Delta)^c$, we have
\begin{equation*}
    k(r)=\alpha C(\sigma)\left(\frac{r}{\sigma}\right)^\nu K_\nu\left(\frac{r}{\sigma}\right)
\end{equation*}
where $C(\sigma)^{-1}=2^{c-1}(2\pi)^{\frac{d}{2}}\Gamma(c)\sigma^d$, $\nu=c-d/2$ and $K_\nu$ denotes the modified Bessel function of order $\nu$, see e.g. \cite{adamsFunctionSpacesPotential1999,rasmussenGaussianProcessesMachine2008,micheliMatrixvaluedKernelsShape2014}. These kernels are denoted Mat\'ern or Bessel kernels. To make the value at $r=0$ be equal to $\alpha$ as for the squared exponential kernel, one can use the $r\to 0$ limit value of $K_\nu(r)$ being $2^{\nu-1}\Gamma(\nu)r^{-\nu}$ for $\nu>0$ and redefine $k$ with the changed constants
\begin{equation}
    \kMat(r)=\alpha\frac{2^{1-\nu}}{\Gamma(\nu)}\left(\frac{r}{\sigma}\right)^\nu K_\nu\left(\frac{r}{\sigma}\right)
    .
    \label{eq:laplace_kernels}
\end{equation}
Letting $\hat{k}$ denote the Fourier symbol of $k$, we have $\hat{k}(\omega)=\alpha(1+\sigma^2\|\omega\|^2)^{-c}$ and hence $\hat{g}(\omega)=\alpha^2(1+\sigma^2\|\omega\|^2)^{-2c}$ from which we see that 
\begin{equation*}
    \gMat(r)=\alpha^2 2^{\frac{3d}{2}-2c+1}\pi^{d/2}\sigma^d\frac{\Gamma(c)^2}{\Gamma(c-\frac{d}{2})^2\Gamma(2c)}\left(\frac{r}{\sigma}\right)^{2c-\frac{d}{2}}K_{2c-\frac{d}{2}}\left(\frac{r}{\sigma}\right)
    \ .
\end{equation*}
The kernels are plotted for different values of $\sigma$ and $c$ in Figure~\ref{fig:kernels}.
\begin{figure}[h]
    \centering
    \includegraphics[width=0.32\textwidth]{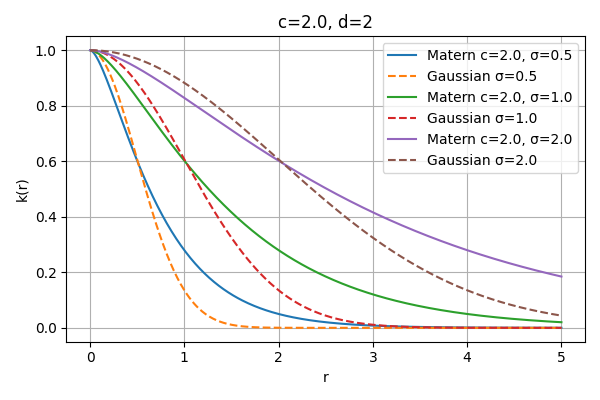}
    \includegraphics[width=0.32\textwidth]{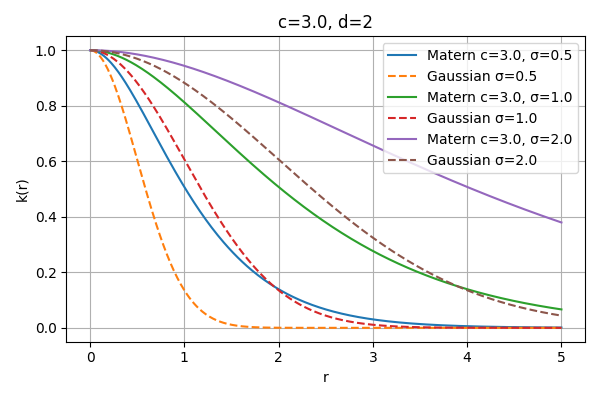}
    \includegraphics[width=0.32\textwidth]{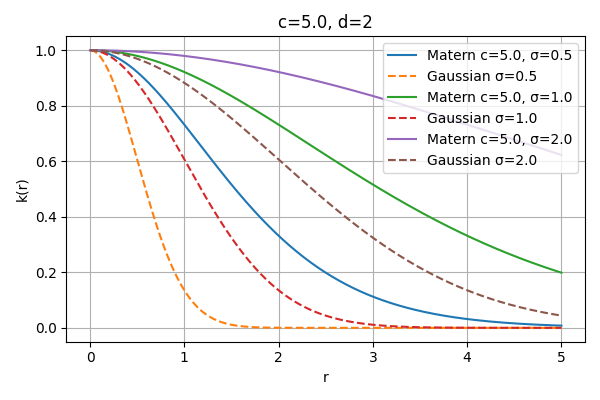}
    \caption{Kernels $\kSqE$ and $\kMat$ for different values of $\sigma$ and $c$ with $\alpha=1$.}
    \label{fig:kernels}
\end{figure}

The Hilbert space view doesn't directly give the relation between smoothness of $k$ and the smoothness of the generated flow, but it gives a way to consider the Kunita flow as a solution to an SDE in a Hilbert space without directly introducing the machinery of Kunita's results.

The Hilbert space formulation in addition gives a natural way to parametrize the flow applied to curves using other bases, e.g. Fourier space representations. This is developed in \cite{bakerConditioningNonlinearInfinitedimensional2024}. In $\R^2$ with a Fourier basis $\{e_n\}_{n=1}^\infty$, $e_n(x)=e^{inx}$ and $g_{l,m}(x_1,x_2)=e^{ilx_1+imx_2}$, writing the SDE \eqref{eq:kunita_flow_hilbert} into basis elements gives
\begin{equation*}
    dX_t = \sum_{n=1}^{\infty}\sum_{l,m=1}^{\infty} \langle e_n,Q^{1/2}(X_t)(g_{l,m})\rangle dw_{l,m}(t)e_n.
\end{equation*}
One obtains a numerical scheme by truncating the sum to a finite number of Fourier modes and by using the fast Fourier transform.

\subsection{Induced shape processes}
Kunita flows of diffeomorphisms induce stochastic processes on shape spaces through actions of diffeomorphisms. If $\phi\in\Diff(\R^d)$ acts on a shape $s$ by $\phi.s$, a Kunita flow $\phi_t$ induces a stochastic process on $s$ through the action $s_t=\phi_t.s$.

For example, in the case of landmarks $x=(x_1,\ldots,x_n)$ in $\cS_n$, a Kunita flow results in the finite dimensional process $x_t=\phi_t.x=(\phi_t(x_1),\ldots,\phi_t(x_n))$. The dynamics of this process are solutions to the SDE
\begin{equation}
    dx_t^i
    =\left(Q^{1/2}(X_t)dW_t\right)(x_t^i)
    =
    \int_{D}k(x_t^i,\zeta)dW_t(\zeta)d\zeta
    \label{eq:kunita_flow_landmark}
\end{equation}
where $x_t^i$ denotes the $i$th landmark of the configuration $x_t$ and $X_t=\phi_t-\Id_d$.

%As above, local characteristics $\alpha$ of $\phi_t$ is here equal to $g$. 
The covariance matrix of the infinitesimal steps of $dx_t$ are then given by the matrix $G(x_t)=[g(x_t^i,x_t^j)]^i_j$ with $g$ from \eqref{eq:kunita_flow_hilbert_operator_square}. If we let $G(x_t)^{1/2}$ be a square root of $G(x_t)$, the stochastic increments $dx_t$ can equivalently be written directly in the landmark space as $G(x_t)^{1/2}dW_t$ where $W_t$ is now an $nd$-dimensional Brownian motion.

\subsection{Kunita flows and shape stochastics}
We now prove that processes on shape spaces induced by Kunita flows satisfy the properties (\ref{item:representation_consistency})-(\ref{item:discretization_recovery}) of section~\ref{sec:axiomatic_approach}. The results follow almost directly from the construction of the Kunita flow and Kunita's results.

\begin{proposition}
    Let $\phi_t$ be a Kunita flow with local characteristics $a_t,b_t$ such that $\phi_t$ is a flow of $C^l(\R^d)$-diffeomorphisms. Let $\phi\in C^l(\R^d)$ act on the shape space $\cS$ by $\phi.s_0$, $s_0\in\cS$. Then $\phi_t$ induces a shape process $s_t=\phi_t.s_0$ satisfying the properties (\ref{item:representation_consistency}) and (\ref{item:shape_structure_preservation}). If the local characteristics are invariant to rigid motions, then property (\ref{item:rigid_transformation_invariance}) holds.
\end{proposition}
\begin{proof}
    By construction of the Kunita flow and the action, $\phi_t$ induces a process $s_t=\phi_t.s$ where, for all $t$, $s_t\in\cS$. It is therefore shape structure preserving (property (\ref{item:shape_structure_preservation})).

    The induced shape process $s_t$ is by construction a pushforward of the Kunita flow $\phi_t$ which can be regarded a stochastic process on a shape space itself. Because $\phi_t$ is defined without using properties of $\cS$, $s_t$ is independent of the representation (property \ref{item:representation_consistency}).

    If the local characteristics are translation and rotation invariant then $\phi_t$ is equivariant to the action of the translation and rotation groups on the initial value. This is inherited by the induced shape process $s_t=\phi_t.s$ (property \ref{item:rigid_transformation_invariance}).
\end{proof}

We now turn to discretization consistency. The law of a Kunita flow is determined by its finite dimensional distributions. This is equivalent to its motion being determined by the motion of any finite set of landmarks, i.e. that its action $\phi_t.s$ is known for all $s\in\cS_n$ for the landmark configuration space $\cS_n$ and all $n$. However, since the flow is characterized by its local characteristics, the law of the Kunita flow is actually determined only by its 2-point motions, i.e. if its action on $\cS_2$ is known, see \cite[section 4.2]{kunitaStochasticFlowsStochastic1997}. Referring to the discretization property (\ref{item:discretization_recovery}), we see that the law of $\phi$ can be recovered if the family of finite dimensional discretizations $\mathcal P$ includes all 2-point pairs on $\R^d$.

Often, a shape does not cover all points of $\R^d$, i.e. a curve is only a subset of $\R^d$. If the landmarks represent a discretization of such a subset of $\R^d$, we can require $\mathcal P$ to include all 2-points pairs in the subset but not all of $\R^d$. In this case, the law of $\phi_t$ can be recovered for any point in the subset, but it does not allow to recover $\phi_t$ outside the subset. To be precise, we let $s:M\to \R^d$ be a parametrized subset of $\R^d$ and let the discretization family $\mathcal P$ include all points in $s(M)$. Kunita's results then imply the following discretization consistency property:
\begin{proposition}
    Let the discretization family $\mathcal P$ include all 2-point pairs in $s(M)$. Then the law of $s_t=\phi_t.s_0$, $s_0\in\cS$, is determined by the law of the discretized processes $s_{P,t}=f_P(s_t)$ for all $P\in\mathcal P$.
\end{proposition}
Therefore, with the assumptions of the proposition, the discretization property (\ref{item:discretization_recovery}) holds for the shape process $\phi_t.s$.

\subsection{Shape variance}
Kunita flows provide a notion of variance for shape processes with the unit-normalized trace of the covariance operator $Q$
\begin{equation*}
    \var(Q)
    =\lim_{|D|\to\infty}\frac{1}{|D|}\Tr(Q_D)
    =\lim_{|D|\to\infty}\frac{1}{|D|}\int_{D}g(\zeta,\zeta)d\zeta
\end{equation*}
where $D$ represents bounded domains in $\R^d$ with volume $|D|$ and $Q_D$ the covariance operator restricted to $D$. This notion stems from the trace of the covariance matrix in finite dimensions, but it is not tied to the specific representation of the shape as is for example of the case when variance is measured on landmark configurations and thus dependent on the specific number of landmarks. The normalization for volume reflects the unboundedness of $\R^d$.

With the kernel being translation invariant, we have 
\begin{equation*}
    \var(Q)
    =
    \lim_{|D|\to\infty}\frac{g(0)}{|D|}\int_{D}d\zeta
    =
    g(0)
\end{equation*}
Plugging in the explicit expression for $\gSqE$ from \eqref{eq:gaussian_kernel}, we get $\var(Q)=\alpha^2\pi^{d/2}\sigma^d$ for the squared exponential kernel. For the Mat\'ern kernels $\kMat$, we have
\begin{equation*}
    \var(Q)
    %=
    %\alpha^2
    %\frac{\Gamma(2c-\frac{d}{2})}
    %{2^d\pi^{d/2}\Gamma(2c)\sigma^d}
    =
    \alpha^2 2^d \pi^{d/2} \sigma^d \frac{\Gamma(c)^2}{\Gamma(\nu)^2} \frac{\Gamma(2c-\frac{d}{2})}{\Gamma(2c)}
\end{equation*}
when $c>d/4$.

In practical applications including the squared exponential and Mat\'ern kernels above, the kernel $k$ equivalently $g$ and hence the variance $\var(Q)$ depend on parameters that can be estimated from observed shape data, e.g. landmark configurations. However, though estimated from specific shape representations, the variance is tied to the shape process itself, and not to the specific representation. We exemplify such estimation in section~\ref{sub:inference}.

To exemplify the notion of variance, Figure~\ref{fig:variance} shows Kunita flows on butterfly wing shapes with different variances using the Mat\'ern kernel with $c=7/2$. The figure shows samples of the shape process induced from the Kunita flow and samples from repeated sampling of the time $t$ distribution $\phi_t.s$ for fixed $t$ and for different parameters and hence values of $\var(Q)$.
\begin{figure}[h]
    \centering
    \includegraphics[width=0.98\textwidth]{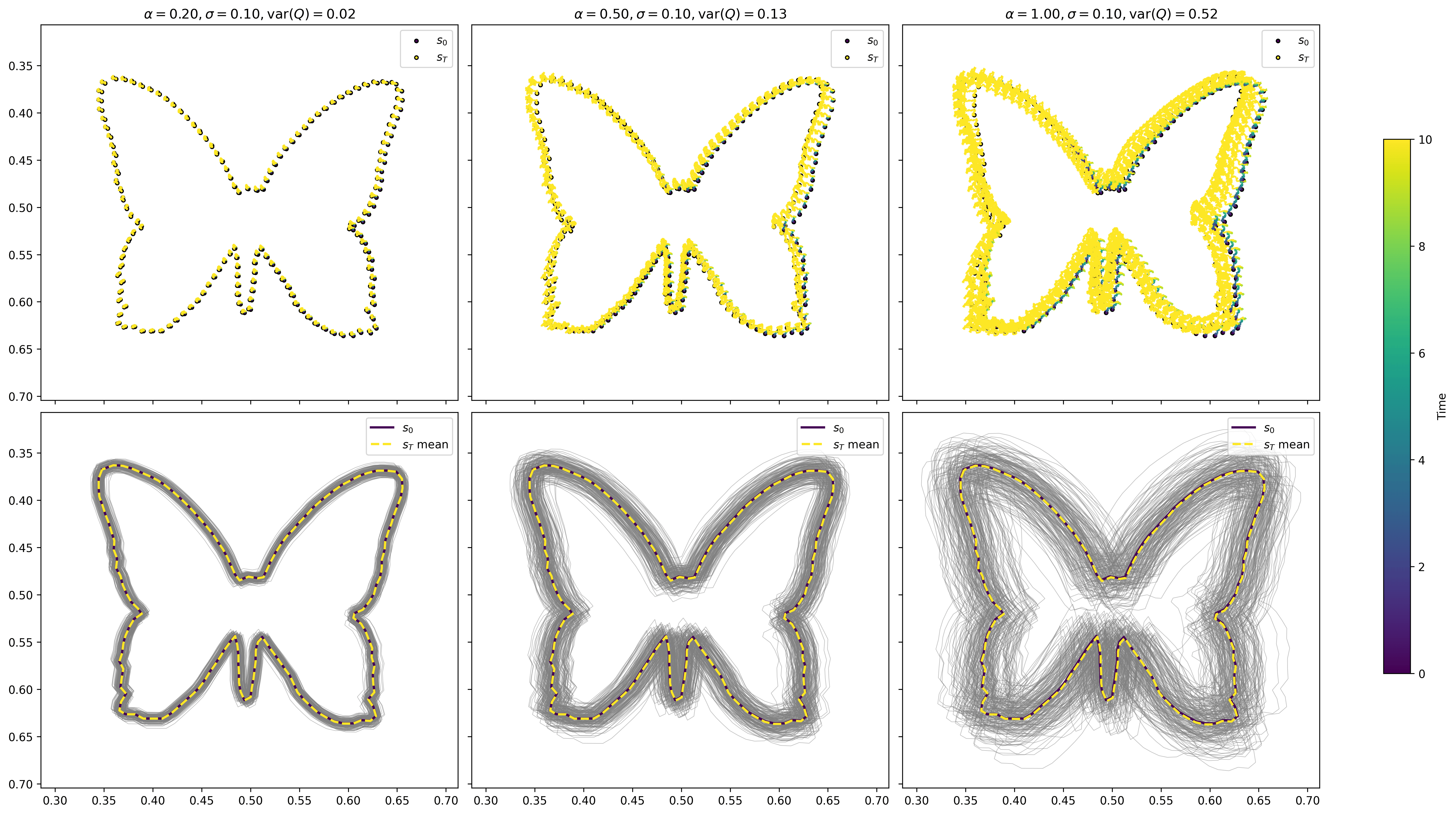}
    \includegraphics[width=0.98\textwidth]{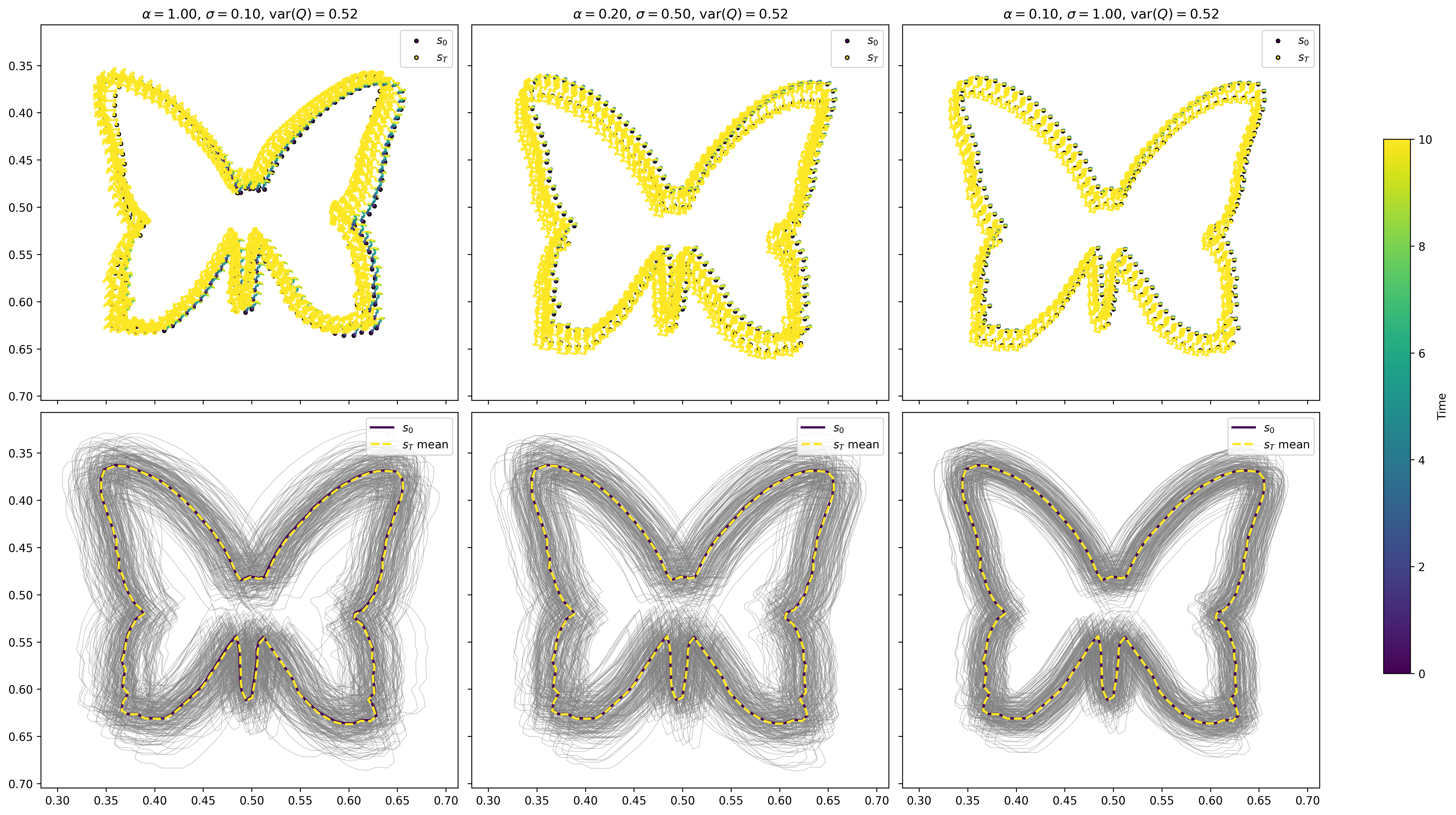}
    \caption{Kunita flows acting on butterfly wing shapes with different combinations of parameters and hence values of $\var(Q)$. Rows 1,3: Samples of the shape process induced from the Kunita flow. Rows 2,4: Samples from repeated sampling of the time $t=10$ distribution $\phi_t.s_0$. Rows 1,2: Different values of $\alpha$ and hence $\var(Q)$ which implies that the total variance in the system to vary. Rows 3,4: Different values of $\sigma$ and with $\alpha$ changed accordingly to keep $\var(Q)$ constant. This keeps the total variance constant while the correlation between landmarks changes.}
    \label{fig:variance}
\end{figure}

\subsection{It\^o vs. Stratonovich}
We have described Kunita flows on It\^o form, but they can equivalently be formulated in Stratonovich form with Stratonovich integral $\circ dW_t$ and with changed drift originating from the It\^o-Stratonovich correction term. Because mappings of Stratonovich processes obey the chain rule, the Stratonovich formulation is sometimes preferable. One example is the stochastic EPDiff shape flows and the fluid dynamics processes from which they arise (\cite{holmVariationalPrinciplesStochastic2015}), where the Stratonovich version is used, see section~\ref{sec:other_stochastic_processes}. Shape flows induced by an It\^o Kunita flow may receive an additional drift term involving the derivative of the action when written as an SDE directly on the shape space, in contrast to Stratonovich flows that, because of the chain rule, will not receive the additional drift. On the other hand, the It\^o formulation can be easier to interpret and implement. 

\section{Other shape stochastic processes}
\label{sec:other_stochastic_processes}
We here survey other related shape stochastic processes and briefly mention other related time series models applicable to shape data, including processes specifically for landmark dynamics, processes for outer shape spaces originating from fluid dynamics, and processes on inner shape representations. In addition to this, several papers focus on theoretical aspects of Kunita flows relevant to applications in shape analysis, including flows of diffemorphisms and Laplace approximation of the transition density (\cite{markussenLaplaceApproximationTransition2009}), large deviation principles (\cite{budhirajaLargeDeviationsStochastic2010}), and most probable paths for Kunita flows (\cite{grongMostProbablePaths2022}). See also \cite{stanevaLearningShapeTrends2017} for construction and statistical estimation of several shape processes.

\subsection{Riemannian Brownian motions}
On finite dimensional Riemannian manifolds, the Riemannian Brownian motion is a fundamental diffusion process. The Riemannian metric gives the Laplace-Beltrami operator $\Delta$, and the Brownian motion is the diffusion process with infinitesimal generator $\frac12\Delta$. In landmark shape spaces $\cS_n$, the corresponding SDE can be written in coordinates
\begin{equation*}
    d x_{t,i} = -\frac{1}{2} \sum_{l,m=1}^n K(x_t)^{l,m} \Gamma(x_t)_{l,m}^i dt + \sqrt{K(q_t)^i} dW
\end{equation*}
where $K(x_t)$ is the kernel matrix $[K(x)]^i_j=k(\|x_i-x_j\|)$ and $\Gamma$ the Christoffel symbols of the Riemannian metrix.

In a shape context, the Riemannian Brownian motion has been considered in for example \cite{sommerBridgeSimulationMetric2017}, where estimation of parameters of the metric from observed data is pursued, and in \cite{habermannLongtimeExistenceBrownian2024} where long-time existence of the process is studied specifically on landmark manifolds. The process is tied to the geometry of the specific landmark space and as such will depend on the number of landmarks, thus it is not defined independently of the specific representation as required by property (\ref{item:representation_consistency}). It is currently an open question if the shape structure preservation property (\ref{item:shape_structure_preservation}) holds for the Riemannian Brownian motion on $\cS_n$ when $n>2$, because ruling out collision of landmarks with this process is non-trivial.

\subsection{Stochastic EPDiff}
In fluid dynamics, a family of stochastic flows introduced in \cite{holmVariationalPrinciplesStochastic2015} include momentum of the fluids with stochastic perturbations of the flow state coupling to the momentum. This results again in a stochastic flow of diffeomorphisms that can act on a shape space as pursued in 
\cite{arnaudonGeometricFrameworkStochastic2019}.

Similarly to the derivation of the EPDiff equations for extremal flows for the LDDMM metric, the stochastic EPDiff equations can be derived from a variational principle, either by perturbing a reconstruction equation by the added stochastics, or by starting from a stochastically perturbed Hamiltonian. The stochastic EPDiff equation preserve many of the properties from the deterministic case, including that the dynamics descend to shape spaces via the action, and that optimal flows on a shape space can be lifted horizontally to flows on the diffeomorphism group. A particular case of this is for landmark configurations where the stochastic EPDiff equations extend the Hamiltonian equations \eqref{eq:landmark_geodesic} from the deterministic case. With an underlying grid of fixed noise fields $\sigma_1,\ldots,\sigma_J$, the stochastic EPDiff landmark equations are given by
\begin{equation*}
    \begin{split}
    dx_{t,i} &= \nabla_{p_{t,i}}H(x_t,p_t)dt + \sum_{j=1}^J \sigma_j(x_{t,i}) \circ dW_{t}^j , \\
    dp_{t,i} &= -\nabla_{x_{t,i}}H(x_t,p_t)dt - \sum_{j=1}^J \nabla_{x_{t,i}} (p_{t,i} \cdot \sigma_j(x_{t,i})) \circ dW_{t}^j .
\end{split}
%\label{eq:epdiff_landmark}
\end{equation*}
The use of the fixed noise fields can be seen as a noise finite-dimensional version of the convolution \eqref{eq:kunita_flow_hilbert_operator}. With zero momentum, the shape EPDiff equations are particular cases of Kunita flows in Stratonovich form.

\subsection{Inner processes}
\cite{trouveShapeSplinesStochastic2012} propose a stochastic model for shape evolution by adding a stochastic term to the momentum equation of the Hamiltonian equations as follows
\begin{equation*}
    \begin{split}
    dx_{t,i} &= \nabla_{p_{t,i}}H(x_t,p_t)dt, \\
    dp_{t,i} &= -\nabla_{x_{t,i}}H(x_t,p_t)dt + dW_{t,i} .
\end{split}
%\label{eq:epdiff_momentum_noise}
\end{equation*}
This can be seen as allowing external random forces acting on the particles. \cite{vialardExtensionInfiniteDimensions2013} extend this to infinite dimensions, by considering what happens as the number of particles tends to infinity, thereby including curves and surfaces in the model. To do this, the Hamiltonian equations are defined in specific Hilbert spaces of functions $M\mapsto\mathbb{R}^d$, $M=\S^m$, $m=1,2$.

\cite{bakerFunctionSpacePerspective2023} alternatively consider the space of shapes to be the Sobolev spaces $H^\nu(\S^m, \mathbb{R}^d)$, $m=1, 2$ and define stochastic processes directly in these Sobolev spaces. In order to implement this, a spherical harmonic basis is used. The initial shape is decomposed into the basis elements, and then each coefficient satisfies an one-dimensional SDE decoupled from the other coefficients. This setup allows for efficient implementation, but at the cost that there is no guarantee on the shapes being embeddings for long time spans.

\subsection{Langevin dynamics for landmark configurations}
In \cite{marslandLangevinEquationsLandmark2017}, the Hamiltonian form of the geodesic equations for landmark configurations is perturbed to form Langevin dynamics. The Hamiltonian equations \eqref{eq:landmark_geodesic} can be perturbed to the Langevin dynamics
\begin{equation*}
    \begin{split}
    dx_{t,i} &= \nabla_{p_{t,i}}H(x_t,p_t)dt \\
    dp_{t,i} &= \left(-\lambda\nabla_{p_{t,i}}H(x_t,p_t)-\nabla_{x_{t,i}}H(x_t,p_t)\right)dt + \sigma dW_{t,i}
\end{split}
%\label{eq:langevin_landmark}
\end{equation*}
where $\lambda$ is a diffusion coefficient and $\sigma$ is a noise amplitude. %Solutions of the landmark Langevin equations can then be lifted to the diffeomorphism group to give a stochastic flow of diffeomorphisms driven by the landmark dynamics.

\subsection{Autoregressive models and functional data analysis}
In addition to continuous-time stochastic processes, time series for manifold or shape valued data can be modeled with discrete-time autoregressive models or Bayesian models. Examples include autoregressive models for manifold data (\cite{xavierGeneralizationARProcesses2006,zhuSphericalAutoregressiveModels2024}) and Bayesian mixed-effects models for time-series data (\cite{schirattiLearningSpatiotemporalTrajectories2015}).

Another line of work models stochastic shape evolutions using functional data analysis (FDA) techniques. Here, observed trajectories are regarded curves on a shape manifold, and variability is captured by imposing probability models directly on such curves. Examples include Riemannian functional principal component analysis that models variation around a mean curve (\cite{daiPrincipalComponentAnalysis2018,linIntrinsicRiemannianFunctional2019}) and time-series for elastic shape models (\cite{suStatisticalAnalysisTrajectories2014}). Compared with such FDA approaches, the Kunita-flow formulation introduces stochasticity at the level of diffeomorphic transport, thereby yielding a mechanistic model of deformation dynamics, whereas the FDA approaches take a more statistical approach.

\section{Conditioning, bridge sampling and statistical inference}
\label{sec:conditioning_bridge_sampling}
For statistics of shapes when there are assumed underlying stochastic process dynamics, it is common to have shape observations for discrete, positive time points. This often leads to the need to condition the process on these observations, and subsequently to allow simulation from the conditioned process, i.e. sampling from the bridge process. Because of the importance of conditioning and bridge sampling in applications, we here survey methods for conditional sampling. Subsequently, we describe one approach for estimating parameters of the process from observed data using sampling of the conditioned process.

In finite dimensions, classical results state that a process $X_t$ resulting from an SDE $dX_t=b(X_t)dt+\sigma(X_t)dW_t$ conditioned on a value $X_T=y$ for a positive time $T$ results in a new process $X_t^y$ with SDE 
\begin{equation*}
    dX_t^y=b(X_t^y)dt+\sigma(X_t^y)\sigma(X_t^y)^* \nabla\log p_{t\to T}(X_t^y,y)dt+\sigma(X_t^y)dW_t
\end{equation*}
where $p_{t\to T}(x,y)$ is the transition density of the process started at $x$ at time $t$ and observed at $y$ at time $T$. The value $y$ would in statistical applications be observed data, e.g. observed shapes. The term $\nabla\log p_{t\to T}(X_t,y)$ is denoted the score of the process. The score is typically not tractable due to the transition density not being tractable. The bridge sampling techniques below use various ways to approximate or identify the score.

\subsection{Score approximations and landmark processes}
\label{sub:score_approximations}
Simulation of conditioned landmark shape processes can be approached by explicit approximations of the score. This is pursued in \cite{arnaudonDiffusionBridgesStochastic2022,sommerBridgeSimulationMetric2017} building on the conditioned diffusion simulation scheme of \cite{delyonSimulationConditionedDiffusion2006}. Here, the score $\nabla\log p_{t\to T}(X_t,y)$ is approximated by the score $-\frac{X_t-y}{T-t}$ of a Brownian motion. Under some assumptions on the process, this gives an approximation of the conditioned process that is absolutely continuous with respect to the true bridge with explicitly computable likelihood ratio. The bridge approximation is then the SDE
\begin{equation*}
    dX_t^\diamond=b(X_t^\diamond)dt-\frac{X_t^\diamond-y}{T-t}dt+\sigma(X_t^\diamond)dW_t
\end{equation*}
A better approximation of the dynamics is available with the guided proposal scheme of van der Meulen and Schauer, see \cite{miderContinuousdiscreteSmoothingDiffusions2021}. Here, the score approximation comes from an auxiliary process that is chosen to match the dynamics of the original process as closely as possible while still having tractable, closed-form transition densities. The score approximation then arises from the auxiliary process. This approach is used in the shape context in \cite{arnaudonDiffusionBridgesStochastic2022}. Denoting the transition density of the auxiliary process by $\tilde{p}_{t\to T}(x,y)$, the guided bridges are given by
\begin{equation}
    dX_t^\circ=b(X_t^\circ)dt+\sigma(X_t^\circ)\sigma(X_t^\circ)^* \nabla\log \tilde{p}_{t\to T}(X_t^\circ,y)dt+\sigma(X_t^\circ)dW_t .
    \label{eq:guided_bridge}
\end{equation}
This approach can be extended from bridges between two points to bridges along a directed acyclic graph, e.g. a tree with observations at the leaves. The underlying theory is developed in \cite{meulenBackwardFilteringForward2025} and applied in a shape context in \cite{stroustrupStochasticPhylogeneticModels2025} with shapes evolving along the branches of a phylogenetic tree. In section~\ref{sub:inference} below, we briefly describe how bridge simulation using \eqref{eq:guided_bridge} can be used to estimate parameters of the process from observed data in this setting.

\subsection{Conditioning in infinite dimensions}
To move from landmarks to continuous shapes, we need to extend the finite dimensional constructions above to infinite dimensions. In the linear case, this has been treated in \cite{goldysOrnsteinUhlenbeckBridge2008,dinunnoSPDEBridgesObservation2023}. However, as we have seen, shape stochastic processes generally depend non-linearly on the states (property (\ref{item:shape_structure_preservation}) in section~\ref{sec:axiomatic_approach}).
For strong solutions of (linear and non-linear) Hilbert space-valued SDEs, \cite{bakerConditioningNonlinearInfinitedimensional2024} developed a Doob's $h$-transform for conditioning processes to end in a given subset. 
Hence viewing the SDE via the Hilbert space formulation, we can use this to construct a conditioned process.

To this end, let $X_t \in L^2(D, \mathbb{R}^d)$ be the Hilbert space formulation of the shape process as in \eqref{eq:kunita_flow_hilbert}. Let $\Gamma \subset L^2(D, \mathbb{R}^d)$ be a non-zero measure subset containing the functions on which we wish to condition. Then if $\mathbb{E}[\delta_{\Gamma}(\tilde{X}_T) \mid \tilde{X}_t]$ is twice Fréchet differentiable,
$X_t$ conditioned on the event $X_T \in \Gamma$ satisfies another stochastic differential equation
\begin{align*}
    d\tilde{X}_t = Q^{1/2}(Q^{1/2})^* \nabla_{\tilde{X}_t} \log \mathbb{E}[\delta_{\Gamma}(\tilde{X}_T) \mid \tilde{X}_t]dt + Q^{1/2}(\tilde{X}_t) dW_t,
\end{align*}
with $\tilde{X}_0=0$, where $\delta$ is the Dirac delta function, see \cite{bakerConditioningNonlinearInfinitedimensional2024}. 
It is an open question as to whether $\mathbb{E}[\delta_{\Gamma}(\tilde{X}_T) \mid \tilde{X}_t]$ is twice differentiable when conditioning shapes processes.

Instead of conditioning directly on the process hitting a set $\Gamma$, we may instead condition on a smooth function, for example modelling normally distributed noise around the target with the function $k^y(x)=k(\|x-y\|)$ with $k$ being e.g. the squared exponential kernel $\kSqE$. Since $\kSqE$ is smooth, $\mathbb{E}[k^y(X_T) \mid X_t]$ will be twice Fréchet differentiable. Similarly to before, this induces a conditioned process
\begin{align*}
d\hat{X}_t = Q^{1/2}(\hat{X}_t)Q^{1/2}(\hat{X}_t)^* \nabla_{\hat{X}_t} \log \mathbb{E}[k^y(\hat{X}_T) \mid \hat{X}_t]dt + Q^{1/2}(\hat{X}_t) dW_t,
\end{align*}
with initial condition $\hat{X}_0=0$.

If more structure is available, e.g. if a base measure and corresponding density is available, then it is possible to condition on points directly, see e.g. \cite{pieper-sethmacherClassExponentialChanges2025}.

\subsection{Inference}
\label{sub:inference}
We now describe how parameters of a shape process can be estimated from observed data using the guided bridge scheme described in section~\ref{sub:score_approximations}. We will use the backwards filtering, forwards guiding methods of \cite{meulenBackwardFilteringForward2025} as implemented in the Hyperiax framework\footnote{\url{https://github.com/ComputationalEvolutionaryMorphometry/hyperiax/}}. This extends the guided proposal scheme described in section~\ref{sub:score_approximations} to trees and directed acyclic graphs. A more detailed description of the methodology can be found in \cite{meulenBackwardFilteringForward2025} and \cite{stroustrupStochasticPhylogeneticModels2025}. We focus on landmark shape spaces $\cS_n$ representing butterfly wing shapes, and we perform inference from simulated data demonstrating how parameters of the model can be retrieved. We refer to the above papers for similar results on real data.

Emulating a phylogenetic context where the shapes exhibit random changes through evolution and where species branches over time, we assume a root butterfly wing shape is affected by a stochastic process that branches according to a simple phylogenetic tree. The stochastic process is induced by a Kunita flow with Mat\'ern kernel. When sampling the forward process, we get observations of the shapes at the leaves of the tree. Denoting such a set of leaf observations by $s_1,\ldots,s_N$, we can use the guided bridge scheme to estimate the parameters of the process. This is done by using the guiding process \eqref{eq:guided_bridge} along each edge of the tree. This forward sampling is interchanged with a backwards filtering step that, for a given sample of parameters, propagates from leaves upwards in the tree the information necessary for guiding the bridge, i.e. the parameters of the auxiliary process $\tilde{p}_{t\to T}(x,y)$ in \eqref{eq:guided_bridge}. The backward filtering follows the backward filtering scheme of \cite{meulenBackwardFilteringForward2025}. These steps are repeated in an MCMC loop that iteratively proposes new values for the parameters $\alpha$ and $\sigma$ of the Mat\'ern kernel in between the forwards sampling and backwards filtering steps.

Figure~\ref{fig:inference} shows the phylogenetic tree, sample butterflies at the leaves, and MCMC trace plots of the estimated parameters together with the true parameters. It can be seen that the MCMC chain converges to the true parameter values allowing for estimation of the parameters only from observing the values at the leaves of the tree.

\begin{figure}[h]
    \centering
    \includegraphics[width=0.34\textwidth]{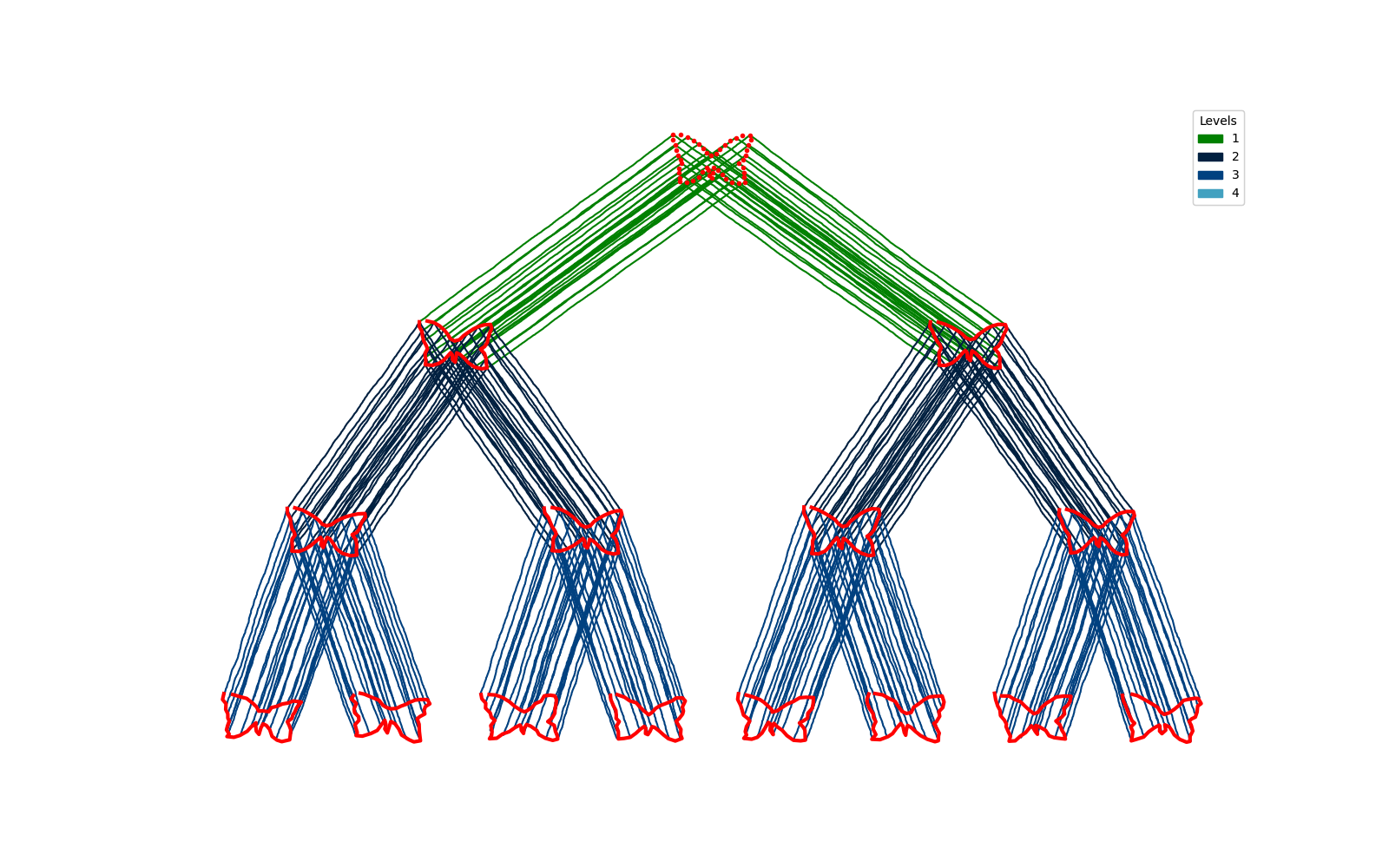}
    \includegraphics[width=0.19\textwidth,clip,trim={150 20 100 0}]{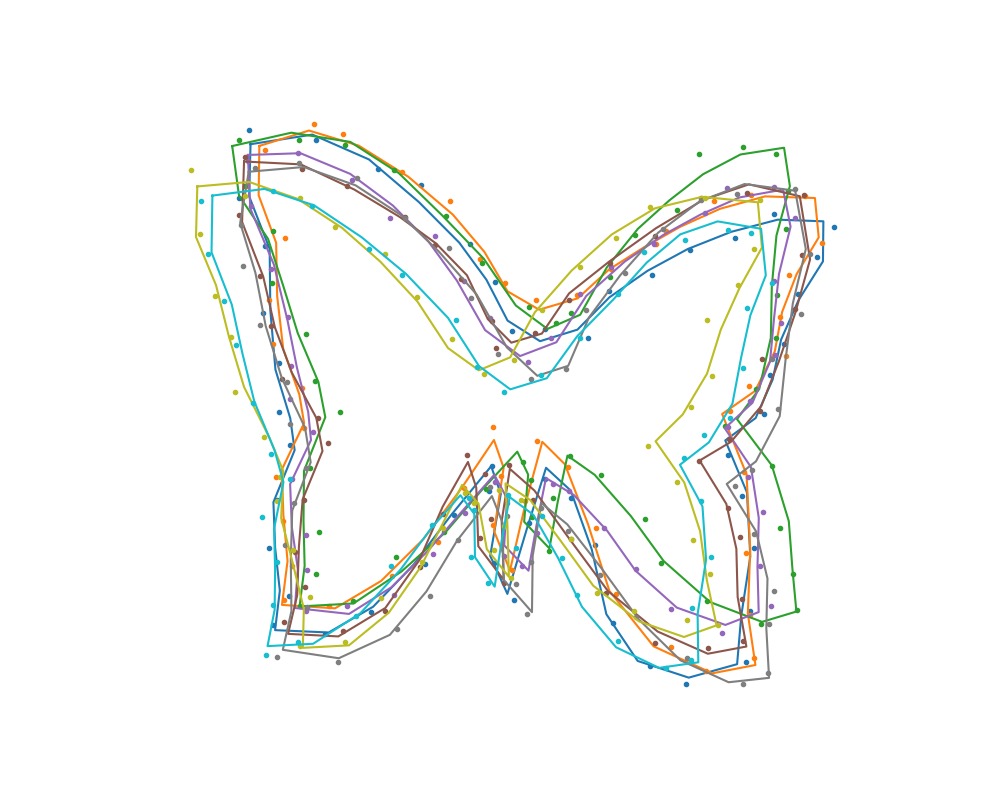}
    \includegraphics[width=0.39\textwidth,clip,trim={0 0 364 0}]{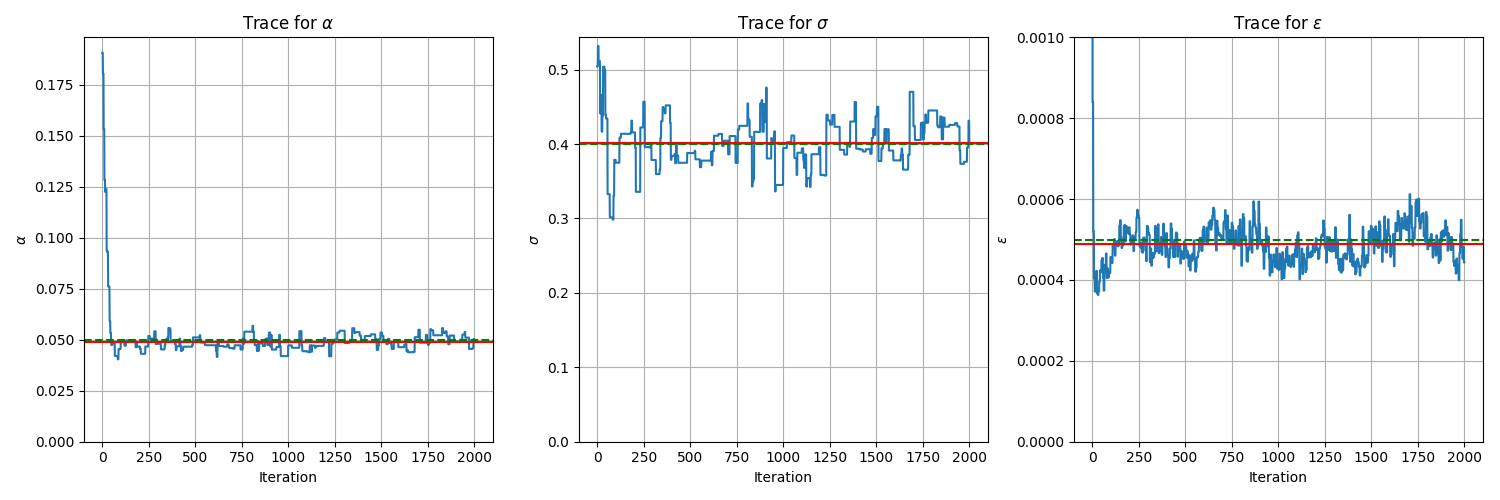}
    \caption{(Left) Example phylogenetic tree with butterfly wing shapes observed at the leaves. (Center) Observed leaf shapes (outlines) with observation noise on landmarks (dots) (Right) MCMC trace plots of the estimated parameters of the kernel $\kMat$, mean of the estimated parameters after burn-in (horizontal solid red lines), and true values (horizontal dashed green lines).}
    \label{fig:inference}
\end{figure}

\section*{Acknowledgements}
The work presented in this paper was supported by the Villum Foundation Grant 40582, and the Novo Nordisk Foundation grants NNF18OC0052000, NNF24OC0093490 and NNF24OC0089608.

%\bibliographystyle{alpha}
%\bibliography{references}
\printbibliography

\end{document}